\documentclass[12pt]{article}
\usepackage{amsmath}
\usepackage{amsthm}
\usepackage{amsfonts}

\newcommand{\mmp}{\mathbb{P}}

\newcommand{\od}{\overset{d}{=}}
\newcommand{\dod}{\overset{d}{\to}}
\newcommand{\tp}{\overset{{\mathbb P}}{\to}}

\newcommand{\me}{\mathbb{E}}
\newcommand{\mr}{\mathbb{R}}
\newcommand{\mn}{\mathbb{N}}

\newcommand{\wwp}{P}

\newcommand{\lin}{\underset{n\to\infty}{\lim}}

\newcommand{\lix}{\underset{x\to\infty}{\lim}}
\newcommand{\lit}{\underset{t\to\infty}{\lim}}

\newtheorem{thm}{Theorem}[section]
\newtheorem{lemma}[thm]{Lemma}

\newtheorem{example}[thm]{Example}
\newtheorem{cor}[thm]{Corollary}

\theoremstyle{definition}

\theoremstyle{remark}
\newtheorem{rem}[thm]{Remark}
\begin{document}
\title{Limit theorems for the number of occupied boxes in the  Bernoulli sieve}
\date{\today}

\author{Alexander Gnedin\footnote{
Department of Mathematics, Utrecht University, Postbus 80010, 3508
TA Utrecht, The Netherlands, e-mail: A.V.Gnedin@uu.nl},\quad
Alexander Iksanov\footnote{ Faculty of Cybernetics, National T.
Shevchenko University of Kiev, 01033 Kiev, Ukraine, e-mail:
iksan@unicyb.kiev.ua}\quad and\quad Alexander Marynych\footnote{
Faculty of Cybernetics, National T. Shevchenko University of Kiev,
01033 Kiev, Ukraine, e-mail: marinich@voliacable.com}}

\maketitle
\begin{abstract}
\noindent
The Bernoulli sieve is a version of the classical  `balls-in-boxes' occupancy scheme,
in which random frequencies of infinitely many boxes are produced by a multiplicative
renewal process, also known as the residual allocation model or stick-breaking.
We focus on the number $K_n$ of
boxes occupied by at least one of $n$ balls, as $n\to\infty$.
A variety of limiting distributions for $K_n$ is derived from the
properties of associated  perturbed random walks.
Refining the approach based on the standard renewal theory
we remove a moment constraint to cover the cases left open in previous studies.

\end{abstract}


\section{Introduction and main result}
In a classical occupancy scheme  balls are thrown independently in
an infinite series of boxes with probability $p_k$ of hitting box
$k=1,2,\dots$, where $(p_k)_{k\in\mn}$ is a fixed collection of
positive frequencies summing up to unity. A quantity of
traditional interest is the number $K_n$ of boxes occupied by at
least one of $n$ balls. In concrete applications `boxes'
correspond to distinguishable species or types, and $K_n$ is the
number of distinct species represented in a random sample of size
$n$. Starting from Karlin's fundamental paper \cite{Karlin}, the
behaviour of  $K_n$ was studied by many authors \cite{Barbour,
Dutko, HwangJanson, Mirakhmedov}. In particular, it is known that
the limiting distribution of $K_n$ is normal if the variance of
$K_n$ goes to infinity with $n$, which holds when $p_k$'s have a
power-like  decay, but does not hold when $p_k$'s decay
exponentially as $k\to\infty$ \cite{BGY}. See \cite{BarGn,
GneHanPit} for survey  of recent results on the infinite
occupancy.

Less explored are the mixture models in which frequencies are
themselves random variables $(P_k)_{k\in\mn}$, while the balls are
allocated independently conditionally given the frequencies. The
model is important in many contexts related to sampling from
random discrete distributions, and may be interpreted as the
occupancy scheme in random environment. The variability of the
allocation of balls is then affected by both the randomness in
sampling and the randomness of the environment. With respect to
$K_n$,
 the environment  may be called {\it strong} if the randomness in $(P_k)$ has dominating effect.
 One way to capture this idea is to consider
the conditional expectation
$$R^\ast_n:={\mathbb E}(K_n\,|\,(P_k))=\sum_{k=1}^\infty {\mathbb E}(1-(1-P_k)^n)$$
and to compare fluctuations of $K_n$ about $R^\ast_n$ with
fluctuations of $R^\ast_n$. By Karlin's \cite{Karlin} law of large
numbers, we always have $K_n\sim R^\ast_n$~{\rm a.s.} (as
$n\to\infty$) so the environment may be regarded as strong if the
sampling variability is negligible to the extent that $R^\ast_n$
and $K_n$, normalized by the same constants, have the same
limiting distributions, see \cite{GPYI} for examples.

In this paper we focus on the limiting distributions of $K_n$ for
the Bernoulli sieve \cite{Gne, GIR, GINR}, which is the infinite
occupancy scheme with random frequencies
\begin{equation}\label{17}
\wwp_k:=W_1W_2\cdots W_{k-1}(1-W_k), \ \ k\in\mn,
\end{equation}
where $(W_k)_{k\in\mn}$ are independent copies of a random
variable $W$ taking values in $(0,1)$.
From a viewpoint,  $K_n$ is the number of blocks of regenerative composition structure
 \cite{BarGnSlow, GPYI} induced by a compound Poisson process with jumps $|\log W_k|$.
Discrete probability distributions with random masses (\ref{17})
are sometimes called residual allocation models, the best known
being the instance associated with Ewens' sampling formula when
$W\stackrel{d}{=}{\rm beta}(c,1)$ for $c>0$. Following \cite{Gne,
GINR}, frequencies \eqref{17} can be considered as sizes of the
component intervals obtained by splitting $[0,1]$ at points of the
multiplicative renewal process $(Q_k: k\in\mn_0)$, where
$$Q_0:=1, \ \ Q_j:=\prod_{i=1}^j W_i, \ \ j\in\mn.$$ Accordingly, boxes can be identified with
open intervals $(Q_k, Q_{k-1})$, and balls with points of an
independent sample $U_1,\ldots,U_n$ from the uniform distribution
on $[0,1]$ which is independent of $(Q_k)$. In this representation
balls $i$ and $j$ occupy the same box iff points $U_i$ and $U_j$
belong to the same component interval.

Throughout we assume that the distribution of $|\log W|$ is
non-lattice, and use the following notation for the moments
$$\mu:=\me |\log W|, \ \ \sigma^2:={\rm Var}\,(\log W) \ \ \text{and} \ \ \nu:=\me |\log (1-W)|,$$
which may be finite or infinite.

Under the assumptions $\nu<\infty$ and $\sigma^2<\infty$, the CLT
for $K_n$ was proved in \cite{Gne} by using the analysis of random
recursions. Assuming only that $\nu<\infty$ in \cite{GINR} a
criterion of weak convergence for $K_n$ was derived from that for
\begin{eqnarray}\label{nx} \rho^\ast(x):=\inf \{k\in\mn: W_1\dots
W_k< e^{-x}\}, \ \ x\geq 0.
\end{eqnarray}
%
In this  paper we  treat the cases of finite and infinite $\nu$ in
a unified way, and  obtain the limiting distribution of $K_n$ directly
from the properties of the counting process
\begin{eqnarray}\label{nx}
N^\ast(x)&:=&\#\{k\in\mn: \wwp_k\geq e^{-x}\}\\&=&\#\{k\in\mn:
W_1\cdots W_{k-1}(1-W_k)\geq e^{-x}\},~~~  x>0,
\end{eqnarray}
in the range of small frequencies (large $x$). Although this
approach is familiar from \cite{BarGn, Karlin}, the application to
the Bernoulli sieve is new. We emphasize here  that the connection
between $K_n$ and $N^\ast(x)$ remains veiled unless we consider
the Bernoulli sieve as the occupancy scheme with random
frequencies (a random environment), and the process of occupancy
counts $K_n$ is analyzed conditionally on the environment. Thus we
believe that the present paper offers a natural way to study the
occupancy problem, since the method is based on a direct analysis
of frequencies and calls for generalizations. Our main result is
the following theorem.
\begin{thm}\label{main}
If there exist functions $f: \mr_+\to\mr_+$ and $g:\mr_+\to\mr$
such that $(\rho^\ast(x)-g(x))/f(x)$ converge weakly (as
$x\to\infty$) to some non-degenerate and proper distribution, then
also $(X_n-b_n)/a_n$ converge weakly (as  $n\to\infty$) to the
same distribution, where $X_n$ can be either $K_n$ or $N^\ast(\log
n)$, and the constants are given by
$$b_n=\int_0^{\log n} g(\log
n-y)\,\mmp\{|\log (1-W)|\in {\rm d}y\}, ~~~a_n=f(\log n).$$
\end{thm}

\noindent
In more detail, we have the following characterization of possible limits.
\begin{cor}\label{main2}
The assumption of {\rm ~Theorem ~\ref{main}~} holds iff either the
distribution of $|\log W|$ belongs to the domain of attraction of
a stable law, or the function $\mmp\{|\log W|>x\}$ slowly varies
at $\infty$. Accordingly, there are five possible types of
convergence:
\begin{enumerate}
\item[\rm (a)]
If $\sigma^2<\infty$ then, with
\begin{equation}\label{26}
b_n=\mu^{-1}\bigg(\log n-\int_0^{\log n}\mmp\{|\log (1-W)|>x\}{\rm
d}x\bigg)
\end{equation}
and $a_n=(\mu^{-3}\sigma^2\log n)^{1/2}$, the limiting
distribution of $(K_n-b_n)/a_n$ is standard normal.
\item[\rm (b)]
If $\sigma^2=\infty$, and
$$\int_0^x y^2\, \mmp\{|\log W|\in {\rm d}y\} \ \sim \ L(x) \ \ x\to \infty,$$ for some $L$ slowly varying at $\infty$,
then, with $b_n$ given in \eqref{26} and $a_n=\mu^{-3/2}c_{[\log
n]}$, where $(c_n)$ is any positive sequence satisfying
$\lin\,nL(c_n)/c_n^2=1$, the limiting distribution of  $(K_n-b_n)/a_n$ is standard
normal.
\item[\rm (c)]
If
\begin{equation}\label{domain1}
\mmp\{|\log W|>x \} \ \sim \ x^{-\alpha}L(x), \ \ x\to \infty,
\end{equation}
for some $L$ slowly varying at $\infty$ and $\alpha\in (1,2)$
then, with $b_n$ given in \eqref{26} and
$a_n=\mu^{-(\alpha+1)/\alpha}c_{[\log n]}$, where $(c_n)$ is any
positive sequence satisfying $\lin \,nL(c_n)/c_n^\alpha=1$, the
limiting distribution of  $(K_n-b_n)/a_n$ is $\alpha$-stable with characteristic
function
$$t\mapsto \exp\{-|t|^\alpha
\Gamma(1-\alpha)(\cos(\pi\alpha/2)+i\sin(\pi\alpha/2)\, {\rm
sgn}(t))\}, \ t\in\mr.$$

\item[{\rm (d)}] Assume that the
relation \eqref{domain1} holds with $\alpha=1$. Let $r:\mr \to\mr$
be any nondecreasing function such that $\lix x\mmp\{|\log
W|>r(x)\}=1$ and set $$m(x):=\int_0^x \mmp\{|\log W|>y\}{\rm d}y,
\ \ x>0.$$ Then, with $$b_n:=\int_0^{\log n}{\log n-y\over
m(r((\log n-y)/m(\log n-y)))}\,\mmp\bigg\{\bigg|\log(1-W)\bigg|\in
{\rm d}y \bigg\}$$ and $$a_n:={r(\log n/m(\log n))\over m(\log
n)},$$ the limiting distribution of  $(K_n-b_n)/a_n$ is $1$-stable with characteristic
function
\begin{equation*}\label{stable}
t\mapsto \exp\{-|t|(\pi/2-i\log|t|\,{\rm sgn}(t))\}, \ t\in\mr.
\end{equation*}

\item[{\rm (e)}] If the relation \eqref{domain1} holds for $\alpha \in [0,1)$
then, with $b_n\equiv 0$ and $a_n:=\log^\alpha n/L(\log n)$, the
limiting distribution of  $K_n/a_n$ is the Mittag-Leffler law $\theta_\alpha$ with moments

$$\int_0^\infty x^k \ \theta_\alpha({\rm d}x)\ = \ {k!\over
\Gamma^k(1-\alpha)\Gamma(1+\alpha k)}, \ \ k\in\mn.$$
\end{enumerate}
\end{cor}

Define  $I_n$ to be the index of the last occupied box, which is
the value of $k$ satisfying $Q_k<\min(U_1,\dots,U_n)<Q_{k-1}$, and
let $L_n:=I_n-K_n$ be the number of empty boxes with indices not
exceeding $I_n$. From \cite{GINR} we know that the number $L_n$ of
empty boxes is regulated by $\mu$ and $\nu$ via the relation
$\lim_{n\to\infty}{\mathbb E}L_n=\nu/\mu$ (provided at least one
of these is finite), and that the weak asymptotics of $I_n$
coincides with that of $\rho^\ast(\log n)$, i.e.
 $(I_n-b_n)/a_n$ and $(\rho^\ast(\log n)-b_n)/a_n$ have the same
proper and non-degenerate limiting distribution (if any). In
\cite{Gne, GINR} it was shown that under the condition
$\nu<\infty$ the weak asymptotics of $K_n$ coincides with that of
$I_n$, hence with that of $\rho^\ast(\log n)$. That is to say,
when $\nu<\infty$  the way $L_n$ varies does not affect the
asymptotics of $K_n$ through the representation $K_n=I_n-L_n$.
Clearly, this result is a particular case of Theorem \ref{main}
because when $\nu<\infty$
\begin{equation}\label{27}
\lim_{x\to\infty} {g(x)-\int_0^{x} g(x-y)\,\mmp\{|\log (1-W)|\in
{\rm d}y\}\over f(x)}=0
\end{equation}
(see Remark \ref{rema} for the proof).

Theorem \ref{main} says  that in the
case $\nu=\infty$ the asymptotics of $L_n$ may affect the
asymptotics of $K_n$, and this is indeed the case whenever \eqref{27}
fails, hence a two-term centering of $K_n$ is indispensable.
The following example illustrates the phenomenon.

\begin{example}{\rm
Assume that, for some $\gamma\in (0,1/2)$,
$$\mmp\{W>x\}={1 \over 1+|\log(1-x)|^\gamma}, \ \  x\in [0,1).$$ Then

$$\me \log^2 W<\infty \ \ \text{and} \ \ \mmp\{|\log(1-W)|>x\}\sim
x^{-\gamma}~~~{\rm as}~x\to\infty,$$ and in this case, $$a_n= {\rm
const}\log^{1/2}n \ \ \text{and} \ \ b_n=\mu^{-1}(\log
n-(1-\gamma)^{-1}\log^{1-\gamma} n+o(\log^{1-\gamma}n)).$$ Thus we
see that the second term $b_n-\mu^{-1}\log n$ of centering cannot
be ignored. Moreover, one can check that
$$\me L_n\sim {1\over \mu} \sum_{k=1}^n{\me W^k\over k}\
\sim \ b_n-\mu^{-1}\log n \ \sim \ {1\over \mu (1-\gamma)} \log^{1-\gamma}n,$$
which demonstrates the
substantial contribution of $L_n$.
}
\end{example}

We shall make use of the poissonized version of the occupancy
model, in which balls are thrown in boxes at epochs of a unit rate
Poisson process. The variables associated with time $t$ will be
denoted $K(t), R^\ast(t),$  etc. For instance, the expected number
of occupied boxes within time interval $[0,t]$ conditionally given
$(P_k)$ is
$$R^\ast(t)=\sum_{n=0}^\infty (e^{-t}t^n/n!) R^\ast_n=\sum_{k=1}^\infty {\mathbb E}(1-e^{-tP_k}).$$
The advantage of the poissonized model is that given $(P_k)$ occupation of boxes $1,2,\dots$,
as $t$ varies, occurs by independent Poisson processes of intensities $P_1,P_2,\dots$.

The variable $N^*(x)$ is the number of sites on $[0,x]$ visited by
a perturbed random walk with generic components $|\log
W|,|\log(1-W)|$. We shall develop therefore some general renewal
theory for perturbed random walks, which we believe might be of
some independent interest. The approach based on perturbed random
walks is more general than the one exploited in \cite{GINR} and is
well adapted  to treat the cases $\nu<\infty$ and $\nu=\infty$ in
a unified way.

\section{Renewal theory for perturbed random walks}

\subsection{Preliminaries}
Let $(\xi_k, \eta_k)_{k\in\mn}$ be independent copies of a random
vector $(\xi, \eta)$ with arbitrarily dependent components $\xi>
0$ and $\eta\geq 0$. We assume that the law of $\xi$ is
nonlattice, although extension to the lattice case is possible.
For $(S_k)_{k\in\mn_0}$ a random walk with $S_0=0$ and increments
$\xi_k$, the seqeunce $(T_k)_{k\in\mn}$ with
$$T_k:=S_{k-1}+\eta_k, \ \ k\in\mn,$$ is called a {\it perturbed random
walk} (see, for instance, \cite{ArGl}, \cite[Chapter 6]{Gut2009},
\cite{Iks07}). Since $\underset{k\to\infty}{\lim} T_k =\infty$
a.s., there is some finite number
 $$N(x):=\#\{k\in\mn: T_k\leq x\}, \ \ x\geq 0,$$
of sites visited on the interval $[0,x]$.
Let
\begin{equation}\label{Ragain}
R(x):=\sum_{k=0}^\infty\bigg(1-\exp\left(-x e^{-T_k}\right)\bigg), \ \ x\geq 0.
\end{equation}
Our aim is to find conditions for the weak convergence of,
properly normalized and centered, $N(x)$ and $R(x)$, as
$x\to\infty$.

It is natural to compare $N^*(x)$ with the number of renewals
$$\rho(x):=\#\{k\in\mn_0:
S_k\leq x\}\ = \ \inf\{k\in\mn: S_k>x\}, \ \ x\geq 0.$$ In the
case ${\mathbb E}\eta<\infty$ weak convergence of one of the
variables $(\rho(x)-g(x))/f(x)$  and  $(N(x)-g(x))/f(x)$ (with
suitable $f,g$) implies weak convergence of the other to the same
distribution. Our main focus is thus on the cases when the
contribution of the $\eta_k$'s does affect the asymptotics of
$N(x)$. To our knowledge, in the literature questions about the
asymptotics of perturbed random walks were circumvented  by
imposing an appropriate moment condition which allowed reduction
to $(S_k)$ (see, for instance, \cite[Chapter 6]{Gut2009},
\cite{Hi}, \cite[Theorems 2.1 and 2.2]{RaSam}).

Recall the  following easy observation: for $x,y\geq 0$
\begin{equation}\label{impo}
\rho(x+y)-\rho(x)\overset{a.s.}{\leq} \rho^\prime(x,y)\od \rho(y),
\end{equation}
where $\rho^\prime(x,y):=\inf\{k-N(x)\in\mn: S_k-S_{\rho(x)}>y\}$.
Furthermore, $(\rho^\prime(x,y): y\geq 0)$ is independent of
$\rho(x)$ and has the same distribution as  $(\rho(y): y\geq 0)$.

Denote  $U(x)=\me \rho(x)=\sum_{k=0}^\infty\mmp\{S_k\leq x\}$ the
renewal function of $(S_k)$. From  \eqref{impo} and Fekete's lemma
we have
\begin{equation}\label{estim}
U(x+y)-U(y)\leq C_1y+C_2, \ \ x,y\geq 0,
\end{equation}
for some positive constants $C_1$ and $C_2$.

The next lemma will be used in the proof of Theorem
\ref{28}.
\begin{lemma}\label{con}
If ${\rho(x)-g(x)\over f(x)}$ weakly converges then
\begin{equation}\label{22}
\lix {g(x)-g(x-y)\over f(x)}=0 \ \ \text{locally uniformly in} \
y,
\end{equation}
and, for every $\lambda\in {\mathbb R}$,
\begin{equation}\label{30}
\lix {\int_0^{x} g(x-y)\,{\rm
d}G(y)-\int_0^{x+\lambda} g(x+\lambda-y)\,{\rm d}G(y)\over f(x)}=0,
\end{equation}
for arbitrary distribution function $G$ with $G(0)=0$.
\end{lemma}
\begin{proof}
For fixed  $f$ call $g_1,g_2$ $f$-equivalent if
 $\lix
{g_1(x)-g_2(x)\over f(x)}=0$.
Clearly \eqref{22} is a property of the class of $f$-equivalent functions $g$.

We refer to the list of  possible limiting laws and corresponding
normalizations for $\rho(x)$ \cite[Proposition A.1]{GINR}.
Relation \eqref{22} trivially holds when $g(x)\equiv 0$. It is
known that $g(x)$ cannot be chosen as zero if the law of $\xi$
belongs to the domain of attraction of the $\alpha$-stable law for
$\alpha\in [1,2]$. It is known that for  $\xi$ in the domain of
attraction of a stable law with  $\alpha\in (1,2]$ one can take
$g(x)=x/\me \xi$ which satisfies \eqref{22}.

Thus the only troublesome  case is the stable domain of attraction
for $\alpha=1$. According to \cite[Theorem 3]{Athreya88}, one can
take $$g(x)={x\over m(r(x/m(x)))},$$ where $m(x):=\int_0^x
\mmp\{\xi>y\}{\rm d}y$, and $r(x)$ is any nondecreasing function
such that $\lix x\mmp\{\xi>r(x)\}=1$. The concavity of $m(x)$
implies that $x\mapsto x/m(x)$ is nondecreasing. Thus $x\mapsto
m(r(x/m(x)))$ is nondecreasing too as  superposition of three
nondecreasing functions. Hence, for every $\gamma\in (0,1)$,
$$g(\gamma x)\geq \gamma g(x), \ \ x>0,$$ which readily implies
subadditivity via
$$g(x)+g(z)\geq \bigg({x\over x+z}+{z\over x+z}\bigg)
g(x+z)=g(x+z).$$ Thus,
$$\underset{x\to\infty}{\lim\sup}\,{g(x)-g(x-y)\over f(x)}\leq 0.$$
For the converse inequality for the lower limit it is enough to
choose non-increasing $g$ from the $f$-equivalence class, and by
\cite[Theorem 2]{Athreya88} this indeed can be done by taking
inverse function to $x\mapsto xm(r(x))$.

The stated uniformity of convergence is checked along the same
lines, and \eqref{30} follows from the subadditivity of $b$ and
easy estimates.
\end{proof}

\subsection{The case without centering}

We start with criteria for the weak convergence of $\rho(x)$ and
$R(x)$ in the case when no centering is needed.
\begin{thm}\label{infinite}
For $Y(x)$ any of the variables $\rho(x)$, $N(x)$ or $R(e^x)$ the
following conditions are equivalent:
\begin{itemize}
\item[\rm (a)]
there exists function $f(x): \mr_+\to\mr_+$ such that, as
$x\to\infty$, $Y(x)/f(x)$ weakly converges to a proper and
non-degenerate law,
\item[\rm (b)]
for some $\alpha\in [0,1)$ and
some function $L$ slowly varying at $\infty$
\begin{equation}\label{111}
\mmp\{\xi>x\}\sim x^{-\alpha}L(x), \ \ x\to\infty.
\end{equation}
\end{itemize}
Furthermore, if \eqref{111} holds then the limiting law is the
Mittag-Leffler distribution $\theta_\alpha$, and
one can take $f(x)=x^\alpha/L(x)$.
\end{thm}

The assertion of Theorem \ref{infinite} regarding $\rho(x)$
follows from \cite[Appendix]{GINR}. For the other two variables
the result is a consequence of the following lemma.
\begin{lemma}
We have
$$\lix \,{N(x)\over \rho(x)}=1 \ \ \text{in
probability}$$ and $$\lix\, {R(x)\over \rho(\log x)}=1 \ \
\text{in probability}.$$
\end{lemma}
\begin{proof}
By definition of the perturbed random walk
\begin{equation}\label{5}
\rho(x-y)-\sum_{j=1}^{\rho(x)}1_{\{\eta_j>y\}}\leq N(x)
\leq \rho(x)
\end{equation}
for $0<y<x$.
Clearly, $\rho(x)\uparrow\infty$ a.s. and
\begin{equation}\label{23}
\rho(x-y)
\geq  \rho(x)-\rho^\prime(x-y, y) \ \ \text{a.s.}
\end{equation}
with $\rho^\prime$ as in \eqref{impo}, from which
\begin{equation}\label{88}
{\rho(x-y)\over \rho(x)} \ \tp \ 1, \ \ x\to\infty.
\end{equation}
Finally, by the strong law of large numbers we have
$$\lix {\sum_{j=1}^{\rho(x)}1_{\{\eta_j> y\}}\over
\rho(x)}=\mmp\{\eta> y\} \ \ \text{a.s.}$$ Therefore, dividing
\eqref{5} by $\rho(x)$ and letting first $x\to\infty$ and then
$y\to\infty$ we obtain the first part of the lemma.

As for the second part, we use the representation
\begin{eqnarray}\label{25}
R(x)&=&\int_1^\infty (1-e^{-x/y}){\rm d}\,N(\log y)\nonumber\\
&=&\int_0^x N(\log x-\log y)e^{-y}{\rm d}y-(1-e^{-x})N(0).
\end{eqnarray}
Since $N(x)$ is a.s.\, non-decreasing in $x$ we have, for any
$a<x$,
$$\int_0^x
N(\log x-\log y)e^{-y}{\rm d}y\geq \int_0^a N(\log x-\log
y)e^{-y}{\rm d}y\geq N(\log x-\log a) (1-e^{-a}).$$ Dividing this
inequality by $\rho (\log x)$, sending $x\to\infty$ along with
using \eqref{88} and the already established part of the lemma,
and finally letting $a\to\infty$, we obtain the half of desired
conclusion.

To get the other half, write
\begin{eqnarray}\label{7}
\int_0^x N(\log x-\log y)e^{-y}{\rm d}y&\overset{a.s.}{\leq} &
\rho(\log x)(1-e^{-x})\\&+&\int_0^1 (\rho(\log x-\log y)-\rho(\log
x)) e^{-y}{\rm d}y,\nonumber
\end{eqnarray}
where \eqref{impo}, the inequality $N(x)\leq \rho(x)$ a.s., and
the fact that $\rho(y)$ is a.s. non-decreasing in $y$ have been
used. Since, by \eqref{estim}, $$\me \int_0^1 (\rho(\log x-\log
y)-\rho(\log x)) e^{-y}{\rm d}y \leq \int_0^1 (C_1|\log
y|+C_2)e^{-y}{\rm d}y<\infty,$$ then dividing \eqref{7} by
$\rho(\log x)$ and sending $x\to \infty$ completes the proof.
\end{proof}

\subsection{The case with nonzero centering}

Now we turn to a more intricate case when some
centering is needed. We denote $F(x)$ the distribution function of
$\eta$
and  $U(x)$ the renewal function of
$(S_k)$.

We will see that a major part of the variability of $N(x)$ is
absorbed by the {\it renewal shot-noise} process $(M(x): x\geq 0)$,
where
$$M(x):=\sum_{k=0}^{\rho(x)-1}F(x-S_k), ~~~ x\geq 0.$$
\begin{lemma}\label{shot}
We have
$$\me \bigg(N(x)-M(x)\bigg)^2=\int_0^x  F(x-y)(1-F(x-y)){\rm
d}\,U(y),$$ which implies that, as $x\to\infty$,
\begin{equation}\label{19}
\me \bigg(N(x)-M(x)\bigg)^2\ = \ O\bigg(\int_0^x (1-F(y)){\rm
d}y\bigg)=o(x).
\end{equation}
\end{lemma}
\begin{proof}
For integer $i<j$,
\begin{eqnarray*}
&&\me \bigg(1_{\{S_i\leq x\}}(1_{\{S_i+\eta_{i+1}\leq
x\}}-F(x-S_i))1_{\{S_j\leq x\}}(1_{\{S_j+\eta_{j+1}\leq
x\}}-F(x-S_j))\bigg| (\xi_k, \eta_k)_{k=1}^j\bigg)\\
&=& 1_{\{S_i\leq x\}}(1_{\{S_i+\eta_{i+1}\leq
x\}}-F(x-S_i))1_{\{S_j\leq x\}}\bigg(F(x-S_j)-F(x-S_j)\bigg)=0.
\end{eqnarray*}
Hence,
\begin{eqnarray*}
 \me \bigg(N(x)-M(x)\bigg)^2&=& \me
\bigg(\sum_{k=0}^\infty 1_{\{S_k\leq
x\}}\bigg(1_{\{S_k+\eta_{k+1}\leq x\}}-F(x-S_k)\bigg) \bigg)^2 \\
&=& \me \sum_{n=0}^\infty 1_{\{S_k\leq
x\}}\bigg(1_{\{S_k+\eta_{k+1}\leq x\}}-F(x-S_k)\bigg)^2\\&=& \me
\sum_{k=0}^\infty 1_{\{S_k\leq
x\}}\bigg(F(x-S_k)-F^2(x-S_k)\bigg)\\ &=& \int_0^x
F(x-y)(1-F(x-y)){\rm d}\,U(y).
\end{eqnarray*}


If
$\me \eta<\infty$, then by the key renewal theorem, as $x\to\infty$,
$$\lix \me \bigg(N(x)-M(x)\bigg)^2\ = \ a^{-1} \int_0^\infty F(y)(1-F(y)){\rm
d}y<\infty,$$ where $a:=\me \xi$ may be finite or infinite.
If  $\me \eta=\infty$ and $a<\infty$, a generalization of the key
renewal theorem due to Sgibnev \cite[Theorem 4]{Sgib} yields
$$\me \bigg(N(x)-M(x)\bigg)^2\ \sim \ a^{-1} \int_0^x (1-F(y)){\rm
d}y.$$
Finally, if $\me \eta=\infty$ and $a=\infty$ a modification of
Sgibnev's proof yields
$$\me \bigg(N^*(x)-M(x)\bigg)^2\ =
\ o\bigg(\int_0^x (1-F(y)){\rm d}y\bigg).$$
Thus \eqref{19} follows in any case.
\end{proof}

\begin{thm}\label{28}
If for some random variable $Z$
\begin{equation}\label{1111}
{\rho(x)-g(x)\over f(x)} \ \dod \ Z, ~~~ x\to\infty,
\end{equation}
then also
\begin{equation}\label{20}
{M(x)-\int_0^x g(x-y){\rm d}\,F(y)\over f(x)} \ \dod \ Z, ~~~
x\to\infty,
\end{equation}
\begin{equation}\label{21}
{N(x)-\int_0^x g(x-y){\rm d}\,F(y)\over f(x)} \ \dod \ Z, ~~~
x\to\infty,
\end{equation}
and
\begin{equation}\label{24}
{R(x)-\int_0^{\log x} g(\log x-y){\rm d}\,F(y)\over f(\log x)} \
\dod \ Z, ~~~
x\to\infty.
\end{equation}
\end{thm}
\begin{proof}
Integrating by parts yields
$$M(x)=\int_0^x F(x-y){\rm d}\rho(y)= -F(x)+\int_{0}^x \rho(x-y){\rm
d}F(y),$$ so to prove \eqref{20} it is enough to show  that, as
$x\to\infty$,
$$T(x):=\int_{0}^x {\rho(x-y)-g(x-y)\over f(x)}\,{\rm d}F(y) \ \dod \
Z.$$ For any fixed $\delta\in (0,x)$
we may decompose $T(x)$ as
 $$T_1(x)+T_2(x):=\int_0^\delta
{\rho(x-y)-g(x-y)\over f(x)}\,{\rm d}F(y)+\int_\delta^x
{\rho(x-y)-g(x-y)\over f(x)}\,{\rm d}F(y).$$
From
the proof of Lemma \ref{con} we know that without loss of
generality it can be assumed that $g(x)$ is nondecreasing. Thus,
almost surely,
\begin{eqnarray*}
{\rho(x)-g(x)\over f(x)}F(\delta)&-&{\rho(x)-\rho(x-\delta)\over
f(x)}F(\delta)\\&\leq& T_1(x)\\&\leq& {\rho(x)-g(x)\over
f(x)}F(\delta)+{g(x)-g(x-\delta)\over f(x)}F(\delta).
\end{eqnarray*}
In view of \eqref{impo} and \eqref{22}, we have the convergence
$\underset{\delta\to\infty}{\lim}\,\lix T_1(x)=Z$ in distribution.

For $x>0$ set $$Z_x(t):={\rho(tx)-g(tx)\over f(x)}, ~~~ t\geq 0$$
and $$\mathcal{Z}_x:=(Z_x(t): t\geq 0).$$ \vspace{.2cm}\noindent

We will establish next that
\begin{equation}\label{102}
{\underset{y\in [0,x]}{\sup}\,(\rho(y)-g(y))\over
f(x)}=\underset{t\in[0,1]}{\sup}\,Z_x(t) \ \dod \ \underset{t\in
[0,1]}{\sup}\,Z(t), \ \ x\to\infty,
\end{equation}
and, similarly,
\begin{equation}\label{103}
{\underset{y\in [0,x]}{\inf}\,(\rho(y)-g(y))\over
f(x)}=\underset{t\in[0,1]}{\inf}\,Z_x(t) \ \dod \ \underset{t\in
[0,1]}{\inf}\,Z(t), \ \ x\to\infty.
\end{equation}

\noindent {\sc Case 1}: $g(x)=x/\me \xi$. Then $Z$ is an
$\alpha$-stable random variable for some $\alpha\in (1,2]$
\cite[Proposition A.1]{GINR}. Denote by $\mathcal{Z}=(Z(t):t\geq
0)$ a stable L\'{e}vy process such that $Z(1)\od Z$. Regard
$\mathcal{Z}_x$ and $\mathcal{Z}$ as random elements of Skorohod's
space $D[0,\infty)$ endowed with the $M_1$-topology.

By \cite[Theorem 1b]{Bing73},
\begin{equation}\label{101}
\mathcal{Z}_x\Rightarrow \mathcal{Z}, \ \ x\to\infty.
\end{equation} Since the
supremum functional is $M_1$-continuous, we obtain (\ref{102}) and (\ref{103})
using the
continuous mapping theorem.

\vspace{.2cm}\noindent {\sc Case 2}: $g(x)\neq x/\me \xi$. Then
$Z$ is a $1$-stable random variable. Set $\mathcal{Z}=(Z(t):t\geq
0)$, where $$Z(t)=Z_1(t)-t\log t, \ \ t\geq 0,$$ and
$(Z_1(t):t\geq 0)$ is a stable L\'{e}vy process such that
$Z_1(1)\od Z$. With this notation we derive \eqref{101} from
\cite[Theorem 2]{Haan}, from which \eqref{102}, \eqref{103} follow
along the above lines.

Now it remains to estimate
\begin{eqnarray*}{\underset{y\in[0,x]}{\inf}(\rho(y)-g(y))\over
f(x)}(F(x)-F(\delta))&\leq&
{\underset{y\in[0,x-\delta]}{\inf}(\rho(y)-g(y))\over
f(x)}(F(x)-F(\delta)) \\&\leq& T_2(x)\\&\leq& {\underset{y\in
[0,x]}{\sup}(\rho(y)-g(y))\over f(x)}(F(x)-F(\delta))\\&\leq&
{\underset{y\in[0,x]}{\sup}(\rho(y)-g(y))\over
f(x)}(F(x)-F(\delta)).
\end{eqnarray*}
Using \eqref{102} and \eqref{103}, we conclude that
$\underset{\delta\to\infty}{\lim}\lix T_2(x)=0$ in probability.
The proof of \eqref{20} is complete.

In view of \eqref{19}, $\me (M(x)-N(x))^2=o(x)$. Since $a^2(x)$
grows not slower than $x$ (see \cite[Proposition A.1]{GINR}),
Chebyshev's inequality yields
$${N(x)-M(x)\over f(x)}\ \tp \ 0, \ \ x\to\infty.$$ Now \eqref{21}
follows from \eqref{20}.

It remains to establish \eqref{24}. To this end, introduce for $x>1$
\begin{eqnarray*}
Q_1(x):=\int_1^x e^{-y}(N(\log x)-N(\log x-\log y)){\rm d}y\geq 0,\\
Q_2(x):=\int_0^1 e^{-y}(N(\log x-\log y)-N(\log x)){\rm d}y\geq 0.
\end{eqnarray*}
Using
$$\me N(x)=\int_0^x F(x-y){\rm
d}\,U(y)=-F(x)+
\int_{0}^x U(x-y){\rm d}\,F(y)$$
and
\eqref{estim}, we conclude that for $y\in (1,x)$,
$$\me N(\log x)-\me N(\log x-\log y)\leq C_1(1+F(0))\log y+C_2(1+F(0)).$$
Therefore, $\me Q_1(x)=O(1)$, as $x\to\infty$, whence
${Q_1(x)\over f(\log x)}\tp 0$. Similarly, ${Q_2(x)\over f(\log
x)}\tp 0$. Thence, recalling \eqref{25}
$${Q_1(x)-Q_2(x)\over f(x)}={(1-e^{-x})N(\log x)-R(x)-(1-e^{-x})N(0)\over f(x)} \
\tp \ 0, \ \ x\to\infty.$$ As $N(\log x)$ grows in probability not
faster than  logarithm, we conclude that $${N(\log x)-R(x)\over
f(\log x)} \ \tp \ 0, \ \ x\to\infty.$$ Now an appeal to
\eqref{21} completes the proof.
\end{proof}

\section{Proof of Theorem \ref{main}
}

Set $$S^\ast_0:=0 \ \ \text{and} \ \ S^\ast_k:=|\log
W_1|+\ldots+|\log W_k|, \ \ k\in\mn,$$ and
$$T^\ast_k:=S^\ast_{k-1}+|\log (1-W_k)|, \ \ k\in\mn.$$ The
sequence $(T^\ast_k)_{k\in\mn}$ is a perturbed random walk. Since
$$\rho^\ast(x)=\inf\{k\in\mn: S^\ast_k>x\}, \ \ N^\ast(\log x):=\#\{k\in\mn: T_k \leq \log
x\},$$ an appeal to Theorem \ref{infinite} (case $g=0$) and to
Theorem \ref{28} (case $g\neq 0$) proves the result for
$N^\ast(\log n)$. To prove the statement for $K_n$ we shall use
the poissonization.

\vspace{.2cm}\noindent {\sc Step 1}.
We first check that
\begin{equation}\label{8}
\underset{t\to\infty}{\lim}\me\, {\rm Var}
(K(t)|(P_k))={\log 2 \over\mu},
\end{equation}
which is $0$ for $\mu=\infty$.
Plainly, this will imply
that
\begin{equation}\label{1}
{K(t) -\me (K(t) |(P_k))\over q(t)} \ \tp \ 0,
\end{equation}
for any function $q(t)$ such that $\lit q(t)=\infty$.

According to \cite[formula (25)]{Karlin},
$${\rm
Var}(K(t)|(P_k))=
\sum_{k=1}^\infty\left(e^{-tP_k}-e^{-2tP_k}\right).
$$
With $U^\ast(x):=\sum_{k=0}^\infty\mmp\{S^\ast_k \leq x\}$ and
$\varphi(t):=\me e^{-t(1-W)}$ we obtain
\begin{eqnarray*}
\me\, {\rm Var}(K(t)|(P_k))&=&\me \sum_{k=1}^\infty
\bigg(\varphi(te^{-S^\ast_{k-1}})-\varphi(2te^{-S^\ast_{k-1}})\bigg)\nonumber\\
&=& \int_0^\infty
\bigg(\varphi(te^{-x})-\varphi(2te^{-x})\bigg){\rm d}U^\ast(x)
\label{eq7},
\end{eqnarray*}
which is the same as
\begin{equation}\label{key1}
\me\, {\rm Var}(K(e^x)|(W_k))=\int_0^\infty A(x-y){\rm
d}U^\ast(x).
\end{equation}
for $A(t):=\varphi(e^t)-\varphi(2 e^t)$, $t\in\mr$.
To proceed, observe that
$$\int_0^\infty {e^{-z(1-W)}-e^{-2z(1-W)}\over z}\,{\rm d}z=\log
2,$$ which implies that $A(t)$ is integrable, since by Fubini's
theorem,
\begin{eqnarray*}\label{aux}
\int_{\mr} A(t){\rm d}t&=&\int_0^\infty
{\varphi(z)-\varphi(2z)\over z}\,{\rm d}z\\ &=& \me \int_0^\infty
{e^{-z(1-W)}-e^{-2z(1-W)}\over z}\,{\rm d}z=\log 2.
\end{eqnarray*}
Furthermore, arguing in the same way as in \cite[Section 5]{GINR}
we can prove that $A(t)$ is directly Riemann integrable.
Therefore, application of the key renewal theorem on $\mr$ to
\eqref{key1} yields \eqref{8}.

Chebyshev's inequality together with \eqref{8} imply that, for every
$\varepsilon>0$,
$$\lit \mmp\{|K(t)-\me (K(t)|(P_k))|>\varepsilon
q(t)|(P_k)\}=0 \ \ \text{in probability},$$ which proves
\eqref{1} upon taking expectation and invoking the Lebesgue
bounded convergence theorem.

\vspace{.2cm}\noindent {\sc Step 2}. Step 1 implies that
${K(t)-g(t)\over f(t)}$ weakly converges to a proper and
nondegenerate probability law if and only if ${\me
(K(t)|(P_k))-g(t)\over f(t)}={R^\ast(t)-g(t)\over f(t)}$ weakly
converges to the same law.

Using this observation, Theorem \ref{infinite} (in case $g=0$) and
formula \eqref{24} of Theorem \ref{28} (in case $g\neq 0$) we
conclude that weak convergence of ${\rho^\ast(x)-g(x)\over f(x)}$
to some distribution $\theta$ implies the weak convergence of both
$${R^\ast(t)-\int_0^{\log t} g(\log
t-y)\,\mmp\{|\log(1-W)|\in {\rm d}y\} \over f(\log t)}$$
and
$${K(t)-\int_0^{\log t} g(\log t-y)\,\mmp\{|\log(1-W)|\in {\rm
d}y\} \over f(\log t)}$$ to $\theta$.

\vspace{.2cm}\noindent {\sc Step 3}. It remains to pass from the poissonized occupancy model
to the fixed-$n$ model. Since
$f(\log t)$ is slowly varying, and in view of \eqref{30},
$$b(t):=\int_0^{\log t} g(\log t-y)\,\mmp\{|\log(1-W)|\in {\rm
d}y\}$$ satisfies
$$\lit {b(t)-b([t(1\pm \varepsilon)])\over
f(\log t)}=0$$ for every $\varepsilon>0$. We thus have
$$X_{\pm}(t):={K(t)-b(\lfloor t(1\pm \varepsilon)\rfloor)
\over f(\log (\lfloor t(1\pm \varepsilon)\rfloor))}\ \Rightarrow \ \theta.$$
Let $C_t$ be the event that the number of balls thrown before time $t$
lies in the limits from $\lfloor(1-\varepsilon)t\rfloor$ to  $\lfloor (1+\varepsilon)t\rfloor\}$.
By monotonicity of $K_n$, we have

\begin{eqnarray*} X_-(t)&\geq&
X_-(t)1_{Z_t}\\&\geq& {K_{\lfloor(1-\varepsilon)t\rfloor}-b(\lfloor t(1-
\varepsilon)\rfloor)\over f(\log (\lfloor t(1-\varepsilon)\rfloor ))}1_{C_t}.
\end{eqnarray*}
Since ${\mathbb P}(C_t) \to  1$, we conclude that
$$\theta(x,\infty)\geq
\underset{n\to\infty}{\lim\sup}\,\mmp\bigg\{{K_n-b(n)\over f(\log
n)}>x\bigg\},$$ for all $x\geq 0$. To prove the converse
inequality for $\liminf$, one has to note that
$$X_+(t)1_{(C_t)^c}\ \ \tp \ 0,$$ and proceed in the same manner.
The proof is complete.

\begin{rem}\label{rema}
Here is the promised verification of \eqref{27}. Below we use
terminology introduced in the proof of Lemma \ref{con}.
\end{rem}
\begin{lemma}\label{tec}
Relation \eqref{27} is a property of the class of $f$-equivalent
functions $g$.
\end{lemma}
\begin{proof}
Assume that $g$ satisfies \eqref{27}. We have to show that any
$g_1$ such that $\lix {g(x)-g_1(x)\over f(x)}=0$ satisfies
\eqref{27}, as well.

Plainly, it is enough to check that
\begin{equation}\label{1212}
A(x):= {\int_0^x (g(x-y)-g_1(x-y)){\rm d}\,F(y)\over f(x)}\ \to 0,
\ \ x\to\infty.
\end{equation}
For any $\varepsilon>0$ there exists $x_0>0$ such that for all
$x>x_0$ ${|g(x)-g_1(x)|\over f(x)}<\varepsilon$. Since $f$ is
regularly varying with index $\beta\in [1/2, 1]$, without loss of
generality we can assume that $f$ is nondecreasing. Hence
\begin{eqnarray*}
|A(x)| & \leq & \int_0^{x-x_0}{|g(x-y)-g_1(x-y)|\over f(x-y)}{\rm
d}\,F(y)\\&+&
\int_{x-x_0}^x {|g(x-y)-g_1(x-y)|\over f(x-y)}{\rm d}\,F(y)\\
        & \leq & \varepsilon + \underset{y\in [0,x_0]}{\sup}\,{|g(y)-g_1(y)|\over f(y)}(F(x)-F(x-x_0)).
\end{eqnarray*}
Sending $x\to\infty$ and then $\varepsilon\downarrow 0$ proves
\eqref{1212}.
\end{proof}
If the law of $|\log W|$ belongs to the domain of attraction of an
$\alpha$-stable law, $\alpha\in (1,2]$ then $(\rho(x)-g(x))/f(x)$
weakly converges with $g(x)=x/\mu$ and appropriate $f(x)$. Such a
$g$ trivially verifies \eqref{27} which, by Lemma \ref{tec},
entails that every $g_1$ from the same $f$-equivalence class
verifies \eqref{27}.

If the law of $|\log W|$ belongs to the domain of attraction of a
$1$-stable law, then $(\rho(x)-g(x))/f(x)$ weakly converges for
$g(x)={x\over m(r(x/m(x)))}$ and $f(x)={r(x/m(x))\over m(x)}$,
with $m$ and $r$ as defined in part (d) of Corollary \ref{main2}.
Since $r$ is regularly varying with index one, without loss of
generality we can assume it and hence $g$ are differentiable.
Since ${g(x)\over xf(x)}$ is regularly varying with index $(-1)$,
it converges to $0$, as $x\to\infty$. Besides that, $\lix
x\mmp\{\zeta>x\}=0$ in view of $\nu<\infty$, where we denoted
$|\log(1-W)|$ by $\zeta$. Hence,
$$\lix {g(x)\mmp\{\zeta>x\}\over f(x)}=0.$$ Thus it suffices to
check that
\begin{equation}\label{tec1}
\lix {\me(g(x)-g(x-\zeta))1_{\{\zeta\leq x\}}\over f(x)}=0.
\end{equation}
Now subadditivity and differentiability of $g$ can be exploited in
order to show that $$|g(x)-g(y)|\leq K|x-y|, \ \ x,y>0,$$ where
$K:=1/m(r(1))$. This immediately implies \eqref{tec1} and the
whole claim by virtue of Lemma \ref{tec}.

\vskip0.1cm \noindent {\bf Acknowledgement} A. Iksanov gratefully
acknowledges the support by a grant from the Utrecht University.

\end{document}